\documentclass[10pt]{amsart}
\usepackage{mathrsfs}
\usepackage{amsmath}
\usepackage{amsfonts}
\usepackage{amssymb}
\usepackage[hidelinks]{hyperref}
\usepackage{subfigure}
\usepackage{color}
\usepackage{fancyhdr}
\usepackage{bbm}
\usepackage{bm}
\usepackage{graphicx}
\usepackage{cite}%²úÉúÈç(ÎÄ[2,4-8])µÄÐ§¹û
\newtheorem{theorem}{Theorem}[section]
\newtheorem{remark}{Remark}[section]
\newtheorem{corollary}{Corollary}[section]
\newtheorem{lemma}{Lemma}[section]
\numberwithin{equation}{section}
\begin{document}

\title{on the divisor problem with congruence conditions}

\author{Lirui Jia}
\address{Department Of Mathematical Sciences, Tsinghua University,
Beijing 100084, People's Republic of China }

\email{jialr@tsinghua.edu.cn}

\author{Wenguang Zhai}
\address{Department of Mathematics, China University of Mining and Thechnology, Beijing 100083, People's Republic of China }

\email{zhaiwg@hotmail.com}

\author{Tianxin Cai}
\address{School of Mathematical Sciences, Zhejiang University,
Hangzhou 310027, People's Republic of China }

\email{txcai@zju.edu.cn}

\thanks{
The first author is supported by the National Natural Science Foundation of China (Grant No. 11871295 and Grant No. 11571303), China
Postdoctoral Science Foundation (Grant No. 2018M631434).  The second author is supported by the National Key Basic Research Program of China (Grant No. 2013CB834201). The third author is supported by the National Natural Science Foundation of China (Grant No. 11571303).}
\date{}{}


\begin{abstract}
Let $d(n; r_1, q_1, r_2, q_2)$ be the number of factorization $n=n_1n_2$ satisfying $n_i\equiv r_i\pmod{q_i}$ ($i=1,2$) and $\Delta(x; r_1, q_1, r_2, q_2)$ be the error term of  the summatory function of $d(n; r_1, q_1, r_2, q_2)$ with $x\geq (q_1q_2)^{1+\varepsilon}, 1\leq r_i\leq q_i$, and $(r_i, q_i)=1$ ($i=1, 2$). We study the power moments and sign changes of $\Delta(x; r_1, q_1, r_2, q_2)$, and prove that for a  sufficiently large constant $C$, $\Delta(q_1q_2x; r_1, q_1, r_2, q_2)$ changes sign in the interval $[T,T+C\sqrt{T}]$ for any large $T$. Meanwhile, we show that for a small constant $c'$,  there exist infinitely many
subintervals of length $c'\sqrt{T}\log^{-7}T$ in $[T,2T]$ where $\pm \Delta(q_1q_2x; r_1, q_1, r_2, q_2)> c_5x^\frac{1}{4}$ always holds.
\end{abstract}

\maketitle

%ÕªÒªÍê±Ï

\section{Introduction}

\subsection{Dirichlet divisor problem}

Let $d(n)$ be the Dirichlet divisor function, $D(x)=\mathop{{\sum}}\limits_{n\leq x}d(n)=\mathop{{\sum}}\limits_{n_1n_2\leq x}1$ be the summatory function. In 1849, Dirichlet proved that
$$D(x)=x\log x+(2\gamma-1)x+O(\sqrt{x}),$$
where $\gamma$ is the Euler constant.

 Let
$$\Delta(x)=D(x)-x\log x-(2\gamma-1)x$$ be the error term in the asymptotic formula for $D(x)$. Dirichlet's divisor problem consists of determining the smallest $\alpha$, for which $\Delta(x)\ll x^{\alpha+\varepsilon}$ holds for any $\varepsilon>0$. Clearly, Dirichlet's result implies that $\alpha\leq\frac{1}{2}$. Since then, there are many improvements on this estimate.
 The best to-date is given by
Huxley\cite{Huxley03,huxley2005exponential}, reads
\begin{equation}\label{huxley}
\Delta(x)\ll x^\frac{131}{416}\log^\frac{26947}{8320}x.
\end{equation}
It is widely conjectured that $\alpha=\frac{1}{4}$ is admissible and is the best possible.

Since $\Delta(x)$ exhibits considerable fluctuations, one natural way to study the upper bounds is to consider the moments.

In 1904, Voronoi \cite{voronoi1904fonction} showed that
\begin{equation*}
  \int_1^T\Delta(x)dx=\frac{T}{4}+O(T^{\frac{3}{4}}).
\end{equation*}
Later, in 1922 Cram\'er\cite{cramer1922zwei} proved the mean square formula
\begin{equation*}
   \int_1^T\Delta(x)^2dx=cT^{\frac{3}{2}}+O(T^{\frac{5}{4}+\varepsilon}),\quad\forall~\varepsilon>0,
\end{equation*}
where $c$ is a positive constant. In 1983, Ivic \cite{ivic1983large} used the method of large values to prove that
\begin{equation}\label{ivic}
    \int_1^T|\Delta(x)|^Adx\ll T^{1+\frac{A}{4}+\varepsilon},\quad\forall~\varepsilon>0
\end{equation}
for each fixed $0\leq A\leq\frac{35}{4}$. The range of $A$ can be extended to $\frac{262}{27}$ by the estimate \eqref{huxley}.
In 1992, Tsang\cite{tsang1992higher}  obtained the asymptotic formula
\begin{equation}\label{01}
  \int_1^T\Delta(x)^kdx=c_kT^{1+\frac{k}{4}}+O(T^{1+\frac{k}{4}-\delta_k}),\quad \text{for}\ k=3,4,
\end{equation}
with positive constants $c_3$, $c_4$, and $\delta_3=\frac{1}{14}$, $\delta_4=\frac{1}{23}$. Ivi\'c and Sargos~\cite{ivic2007higher} improved the values $\delta_3$, $\delta_4$ to $\delta'_3=\frac{7}{20}$, $\delta'_4=\frac{1}{12}$, respectively.
Heath-Brown\cite{heath1992distribution} in 1992 proved that for any positive real number $k<A$, where $A$ satisfies \eqref{ivic}, the limit
$$c_k=\lim_{X\rightarrow\infty}X^{-1-\frac{k}{4}}\int_1^X\Delta(x)^kdx$$
exists.  Then, there followed a series
of investigations on explicit asymptotic formula of the type \eqref{01} for larger values of $k$. In 2004, Zhai \cite{zhai04} established asymptotic formulas for $3\leq k\leq9$.

 At the beginning of the 20th  century, Voronoi\cite{voronoi1904fonction} proved the remarkable exact formula that
  \begin{equation*}
    \Delta(x)=-\frac{2}{\pi}\sqrt{x}\sum_{n=1}^\infty\frac{d(n)}{\sqrt{n}}\big(K_1(4\pi\sqrt{nx})+ \frac{\pi}{2}Y_1(4\pi\sqrt{nx})\big),
  \end{equation*}
  where $K_1$, $Y_1$ are the Bessel functions, and the series on the right-hand side is boundedly convergent for $x$ lying in each fixed closed interval.

Heath-Brown and Tsang \cite{heathbrown1994sign} studied the sign changes of $\Delta(x)$. They proved that for a suitable constant $C > 0$, $\Delta(x)$ changes sign on the interval $[T,T + C \sqrt{T}]$ for every sufficiently large $T$. Here the length $\sqrt{T}$ is almost best possible since they proved that in the interval $[T,2T]$ there are many subintervals of length $\gg\sqrt{T}\log^{-5}T$ such that $\Delta(x)$ does not change sign in any of these subintervals.

\subsection{The divisor problem with congruence conditions}

A divisor function with congruence conditions is defined by
\begin{equation*}
  d(n; r_1, q_1, r_2, q_2) =\mathop{{\sum}}_{\begin{subarray}{c}  n=n_1n_2\\n_i\equiv r_i\!\!\!\!\!\pmod{q_i}\\i=1, 2\end{subarray}}{1},
  \end{equation*}
 of which, the summatory function is
 \begin{equation*}
  D(x; r_1, q_1, r_2, q_2) =\mathop{{\sum}}_{\begin{subarray}{c}  n_1n_2\leq x\\n_i\equiv r_i\!\!\!\!\!\pmod{q_i}\\i=1, 2\end{subarray}}{1}.
\end{equation*}

From~Richert \cite{Richert}, we can find that for $x\geq q_1q_2$, $1\leq r_i\leq q_i$ ($i=1, 2$)
\begin{multline}\label{4}
   D(x; r_1, q_1, r_2, q_2)\\
   = \frac{x}{q_1q_2}\log\Big(\frac{x}{q_1q_2}\Big) - \bigg(\frac{\Gamma'}{\Gamma}\Big(\frac{r_1}{q_1}\Big) + \frac{\Gamma'}{\Gamma}\Big(\frac{r_2}{q_2}\Big) +1\bigg)\frac{x}{q_1q_2}+\Delta(x; r_1, q_1, r_2, q_2).
\end{multline}

From Huxley's estimates\cite{Huxley03}, it follows that
\begin{equation}\label{9a}
  \Delta(x; r_1, q_1, r_2, q_2)\ll\Big(\frac{x}{q_1q_2}\Big)^\frac{131}{416}\Big(\log\Big(\frac{x}{q_1q_2}\Big)\Big)^\frac{26947}{8320}
\end{equation}
uniformly in $1\leq r_1\leq q_1\leq x, 1\leq r_2\leq q_2\leq x$. It is conjectured that
\begin{equation}\label{conjecture}
\Delta(x; r_1, q_1, r_2, q_2)\ll\Big(\frac{x}{q_1q_2}\Big)^{\frac{1}{4}+\varepsilon}
\end{equation}
uniformly in $1\leq r_1\leq q_1\leq x, 1\leq r_2\leq q_2\leq x$, $\forall\varepsilon>0$, which is an analogue of the well-known conjecture that $\Delta(x)\ll x^{\frac{1}{4}+\varepsilon}$.

M\"{u}ller and Nowak\cite{MullerNowak} studied the mean value of $\Delta(x; r_1, q_1, r_2, q_2)$. They pointed out
\begin{equation}\label{muller}
  \int_1^T{\Delta(x; r_1, q_1, r_2, q_2)}dx=\big(\frac{r_1}{q_1}-\frac{1}{2}\big)\big(\frac{r_2}{q_2}-\frac{1}{2}\big)T+O\big( (q_1q_2)^{\frac{1}{4}}T^{\frac{3}{4}}\big),
\end{equation}
and
\begin{equation}\label{muller2}
 \int_1^T{\Delta^2(x; r_1, q_1, r_2, q_2)}dx=c_2(q_1q_2)^{\frac{1}{2}}T^\frac{3}{2}+o\big((q_1q_2)^{\frac{1}{2}}T^\frac{3}{2}\big),
\end{equation}
uniformly in $1\leq r_i\leq q_i\leq T$ $(i=1, 2)$, if $T$ is a large number, and $c_2$ is a constant.

In \cite{Jiazhai2017}, we show that
\begin{equation}\label{pre}
    \int_1^T|\Delta(q_1q_2x; r_1, q_1, r_2, q_2)|^Adx\ll T^{1+\frac{A}{4}}\mathcal{L}^{4A},
\end{equation}
for $0\leq A\leq\frac{262}{27}$ and $T\gg(q_1q_2)^\varepsilon$.

Here we study $\Delta(x; r_1, q_1, r_2, q_2)$ further and give some more results about it.

\textsc{Notations}. For a real number $t$, let $[t]$ be the largest integer no greater than $t$, $\{t\}=t-[t]$, $\psi(t)=\{t\}-\frac{1}{2}$, $\parallel t\parallel=\min(\{t\}$, $1-\{t\})$, $e(t)=e^{2\pi it}$. $\mathbb{C}$, $\mathbb{R}$, $\mathbb{Z}$, $\mathbb{N}$ denote the set of complex numbers, of real numbers, of integers, and of natural numbers, respectively; $f\asymp g$ means that both $f\ll g$ and $f\gg g$ hold. Throughout this paper, $\varepsilon$ denote sufficiently small positive constants, and $\mathcal{L}$ denotes $\log T$.

\vspace{1ex}

\section{Main results}\label{sec:results}

In this paper, we will first discuss the power moments of $\Delta(x; r_1, q_1, r_2, q_2)$ and get the following
\begin{theorem}\label{th:moments}
If $T\gg (q_1q_2)^{\varepsilon}$ is large enough. If $A_{0}>9$ satisfies
\begin{equation*}
  \int_1^{T}{|\Delta(q_1q_2x; r_1, q_1, r_2, q_2)|^{A_0}}dx\ll T^{1+\frac{A_0}{4}+\varepsilon},
\end{equation*}
then for any fixed integer $3\leq k<A_0$, we have
\begin{equation}\label{sk}
  \int_1^{T}{\Delta^k(q_1q_2x; r_1, q_1, r_2, q_2)}dx
  =C_k\int_1^Tx^{\frac{k}{4}}dx+o\big(T^{1+\frac{k}{4}}\big),
\end{equation}
where $C_k\asymp 1$ are  explicit constants.
\end{theorem}

From \eqref{pre}, we can take $A_0=\frac{262}{27}$, which means
\begin{corollary}\label{cor:1}
If $T, r_i$ and $q_i (i=1,2)$ satisfying the hypothesis of Theorem \ref{th:moments}, then \eqref{sk} holds for any fixed integer $3\leq k\leq9$.
\end{corollary}

By using the estimates above, we can get the sign changes of $\Delta(x; r_1, q_1, r_2, q_2)$ as following

\begin{theorem}\label{th:change}
 Let $c_1 > 0$ be a sufficiently small constant and $c_2 > 0$ be a sufficiently
large constant, $q_1\geq 2$, $q_2\geq 3$, $1\leq r_i\leq q_i$ and $(r_i, q_i)=1$ $(i=1, 2)$. For any real-valued function $|f(t)|\leq c_1 t^\frac{1}{4}$, the
function $\Delta(q_1q_2t; r_1, q_1, r_2, q_2)+f(t)$  changes sign at least once in the interval $[T,T + c_2 \sqrt{T}]$ for every sufficiently large $T\gg (q_1q_2)^{\varepsilon}$. In particular, there exist $t_1$, $t_2\in [T,T +c_2 \sqrt{T}]$ such that $\Delta(q_1q_2t_1; r_1, q_1, r_2, q_2)\geq c_1 t_1^\frac{1}{4}$ and $\Delta(q_1q_2t_2; r_1, q_1, r_2, q_2)\leq -c_1 t_2^\frac{1}{4}$.
\end{theorem}

\begin{theorem}\label{th:maintain}
There exist three positive absolute constants $c_3$ ,$c_4$ ,$c_5 $ such that, for
any large parameter $T\gg (q_1q_2)^{\varepsilon}$, and any choice of $\pm$ signs, there are at least $c_3 \sqrt{T} log^{7}T$ disjoint subintervals of length $c_4 \sqrt{T} log^{-7}T$ in $[T,2T]$, such that $\pm \Delta(q_1q_2t; r_1, q_1, r_2, q_2)> c_5t^\frac{1}{4}$, whenever $t$ lies in any of these
subintervals. Moreover, we have the estimate
\begin{equation*}
    meas\big\{t\in[T,2T]:\pm \Delta(q_1q_2t; r_1, q_1, r_2, q_2)> c_5 t^\frac{1}{4}\big\}\gg T.
\end{equation*}
\end{theorem}

We also study the $\Omega$-result of the error term in the asymptotic formula \eqref{sk} for odd $k$ by using Theorem \ref{th:maintain}. Define
\begin{equation*}
    \mathcal{F}_k\big(q_1q_2x; r_1, q_1, r_2, q_2\big):=\int_1^{T}{\Delta^k\big(q_1q_2x; r_1, q_1, r_2, q_2\big)}dx
  -C_kT^{1+\frac{k}{4}}.
\end{equation*}
We have the following
\begin{theorem}\label{th:omega}
 For any $T\gg (q_1q_2)^{\varepsilon}$, the
interval $[T, 2T]$ contains a point $X$, for which
\begin{equation*}
    \mathcal{F}_k\big(q_1q_2X; r_1, q_1, r_2, q_2\big)\gg X^{\frac{1}{2}+\frac{k}{4}}\mathcal{L}^{-7}.
\end{equation*}
\end{theorem}
\begin{remark}
Although at the present moment we can only prove \eqref{sk} for $2\leq k\leq 9$, Theorem \ref{th:omega} holds for any odd $k\geq 2$.
\end{remark}
\vspace{1.5ex}

\section{proof of Theorem \ref{th:moments}}

In this section, we prove Theorem \ref{th:moments} by using  the Voronoi-type formula for $\Delta(x; r_1, q_1, r_2, q_2)$.

\begin{lemma}$($ See \cite{Jiazhai2017} $)$\label{lem:vor}

Let $J=[\frac{\mathcal{L}+2\log q_1q_2-4\log \mathcal{L}}{\log2}]$, $H\geq2$ be a parameter to be determined, and $ T^\varepsilon<y\leq \min(H^2, (q_1q_2)^2 T)\mathcal{L}^{-4}$. Suppose $\frac{T}{2}\leq x\leq T$.
Then
\begin{multline}\label{vor0}
  \Delta(q_1q_2x; r_1, q_1, r_2, q_2)=R_0(x; y)+{R_{12}}(x; y, H)+{R_{21}}(x; y, H)\\+G_{12}(x; H)+G_{21}(x; H)
  +O\big(\log^3(q_1q_2T)\big),
\end{multline}
where
\begin{align}
% to remove numbering (before each equation)
   R_0(x; y)&=\frac {x^{\frac{1}{4}}}{\sqrt{2}\pi}\sum_{n\leq y}\frac{1}{n^{\frac{3}{4}}}\sum_{n=hl}\cos\bigg(4\pi \sqrt{nx}-2\pi\Big(\frac{hr_2}{q_2}+\frac{lr_1}{q_1}+\frac{1}{8}\Big)\bigg), \label{R_0} \\
  \nonumber{R_{12}}(x; y, H)&=\frac {x^{\frac{1}{4}}}{\sqrt{2}\pi}\sum_{y<n\leq 2^{J+1}H^2}\frac{1}{n^{\frac{3}{4}}}\!\!\mathop{{\sum}'}_
{\begin{subarray}{c}  n=hl\\  1\leq h\leq H\\ h\leq l\leq2^{J+1}h\end{subarray}}
\!\!\!\!\cos\bigg(4\pi \sqrt{nx}\!-\!2\pi\Big(\frac{hr_2}{q_2}\!+\!\frac{lr_1}{q_1}\!+\!\frac{1}{8}\Big)\bigg),  \\
\nonumber {R_{21}}(x; y, H)&=\frac {x^{\frac{1}{4}}}{\sqrt{2}\pi}\sum_{y<n\leq 2^{J+1}H^2}\frac{1}{n^{\frac{3}{4}}}\!\!\mathop{{\sum}'}_
{\begin{subarray}{c} n=hl\\  1\leq h\leq H\\ h\leq l\leq2^{J+1}h\end{subarray}}
\!\!\!\!\cos\bigg(4\pi \sqrt{nx}\!-\!2\pi\Big(\frac{hr_1}{q_1}\!+\!\frac{lr_2}{q_2}\!+\!\frac{1}{8}\Big)\bigg),  \\
 \nonumber G_{12}(x; H)&=\sum_{\begin{subarray}{c}  n_1\leq q_1\sqrt{T}\\ n_1\equiv r_1\!\!\!\!\!\pmod{q_1}\end{subarray}}O\biggl(\min\Bigl(1, \frac{1}{H\|\frac{q_1x}{n_1}-\frac{r_2}{q_2}\|}\Bigr)\biggr),\\
\nonumber G_{21}(x; H)&=\sum_{\begin{subarray}{c}  n_2\leq q_2\sqrt{T}\\ n_2\equiv r_2\!\!\!\!\!\pmod{q_2}\end{subarray}}O\bigg(\min\Big(1, \frac{1}{H\|\frac{q_2x}{n_2}-\frac{r_1}{q_1}\|}\Big)\bigg).
\end{align}

\end{lemma}
 where  ${\sum\limits_{n\leq x}}' f(n)$ indicates that if $x$ is an integer, then only $\frac{1}{2}f(x)$ is
counted.

Thus, we can get Theorem \ref{th:moments} by using Lemma \ref{lem:vor} with the approach of Liu \cite{LiuKui}.

\section{Proof of Theorem \ref{th:change}}

In this section, we prove Theorem \ref{th:change} following the approach of \cite{heathbrown1994sign}.

Suppose  $|f(t)|\leq c_1 t^\frac{1}{4}$. Let
\begin{equation*}
    \Delta^{**}(t)=\sqrt{2}\pi t^{-\frac{1}{2}}\Big(\Delta(q_1q_2t^2; r_1, q_1, r_2, q_2)+f(t^2)\Big),\quad \text{for }t\geq1.
\end{equation*}
Define
\begin{equation*}
    K_\zeta(u):=(1-|u|)\big(1+\zeta\sin(4\pi\alpha u)\big)\quad\text{for }|u|\leq1,
\end{equation*}
with $\zeta=1$ or $-1$, and $\alpha>1$ a large number.

\begin{lemma}\label{lem:change}
Suppose $T\gg(q_1q_2)^\varepsilon$ is a  large parameter. Then for each $\sqrt{T}\leq t\leq\sqrt{2T}$, we have
\begin{align*}
   &\int_{-1}^1\Delta^{**}(t+\alpha u)K_\zeta(u)du\\
=&-\frac{\zeta}{2}\sin\bigg(4\pi t -2\pi\Big(\frac{r_2}{q_2}+\frac{r_1}{q_1}+\frac{1}{8}\Big)\bigg) +O(\alpha^{-2})\\
&+O\big(t^{-\frac{1}{2}}\sup_{|u|\leq1}f((t\!+\!\alpha u)^2)\big)\!+\!O\big(t^{-\frac{1}{2}}\mathcal{L}^3\big).
\end{align*}
\end{lemma}
\begin{proof}
Let $J=[\frac{\mathcal{L}+2\log q_1 q_2-4\log \mathcal{L}}{\log2}]$, $H\geq2$ be a parameter to be determined, and $ T^\varepsilon<y\leq \min(H^2, (q_1q_2)^2 T)\mathcal{L}^{-4}$. From \eqref{vor0}, we have
\begin{align}\label{vor}
   \Delta^{**}(t)=&R^*_0(t; y)\!+\!R^*_{12}(t; y, H)\!+\!R^*_{21}(t; y, H)\!+\!\sqrt{2}\pi t^{-\frac{1}{2}}f(t^2)\\
\nonumber&\!+\!O\big(t^{-\frac{1}{2}}\big(G^*_{12}(t; H)\!+\!G^*_{21}(t; H)\big)\big)
  \!+\!O\big(t^{-\frac{1}{2}}\mathcal{L}^3\big),
\end{align}
where
 \begin{align*}
    R^*_0(t; y)&=\sum_{n\leq y}\frac{1}{n^{\frac{3}{4}}}\sum_{n=hl}\cos\bigg(4\pi \sqrt{n}t-2\pi\Big(\frac{hr_2}{q_2}+\frac{lr_1}{q_1}+\frac{1}{8}\Big)\bigg),\\
    R^*_{12}(t; y, H)&=\sum_{y<n\leq 2^{J+1}H^2}\frac{1}{n^{\frac{3}{4}}}\mathop{{\sum}'}_
{\begin{subarray}{c}  n=hl\\  1\leq h\leq H\\ h\leq l\leq2^{J+1}h\end{subarray}}
\cos\bigg(4\pi \sqrt{n}t\!-\!2\pi\Big(\frac{hr_2}{q_2}+\frac{lr_1}{q_1}+\frac{1}{8}\Big)\bigg),\\
  R^*_{21}(t; y, H)&=\sum_{y<n\leq 2^{J+1}H^2}\frac{1}{n^{\frac{3}{4}}}\mathop{{\sum}'}_
{\begin{subarray}{c} n=hl\\  1\leq h\leq H\\ h\leq l\leq2^{J+1}h\end{subarray}}
\cos\bigg(4\pi \sqrt{n}t\! -\!2\pi\Big(\frac{hr_1}{q_1}+\frac{lr_2}{q_2} +\frac{1}{8}\Big)\bigg),\\
      G^*_{12}(t; H) &= \sum_{\begin{subarray}{c}  n_1\leq q_1\sqrt{T}\\ n_1\equiv r_1\!\!\!\!\!\pmod{q_1}\end{subarray}}\min\Bigl(1, \frac{1}{H\|\frac{q_1t^2}{n_1}-\frac{r_2}{q_2}\|}\Bigr),\label{G12}\\
   G^*_{21}(t; H)&=\sum_{\begin{subarray}{c}  n_2\leq q_2\sqrt{T}\\ n_2\equiv r_2\!\!\!\!\!\pmod{q_2}\end{subarray}}\min\Big(1, \frac{1}{H\|\frac{q_2t^2}{n_2}-\frac{r_1}{q_1}\|}\Big).
\end{align*}
Denote
\begin{align*}
    &R^*(t)=R^*_0(t; y)\!+\!R^*_{12}(t; y, H)\!+\!R^*_{21}(t; y, H),&
    &G^*(t)=G^*_{12}(t; H)\!+\!G^*_{21}(t; H).
\end{align*}
Then
\begin{equation}\label{s}
  \Delta^{**}(t)=R^*(t)\!+\sqrt{2}\pi t^{-\frac{1}{2}}f(t^2)\!+\!O\big(t^{-\frac{1}{2}}G^*(t)\big)
  \!+\!O\big(t^{-\frac{1}{2}}\mathcal{L}^3\big).
\end{equation}

We first consider $\int_{-1}^1G^*(t+\alpha u)du$. Noting that
\begin{equation*}
    \min\Big(1,\frac{1}{H\|r\|}\Big)=\sum_{h=-\infty}^\infty a(h)e(hr)
\end{equation*}
with
\begin{align*}
    a(0)\ll H^{-1}\log H,&&a(h)\ll\min\Big(H^{-1}\log H,h^{-2}H\Big),\ h\neq0.
\end{align*}
We have
\begin{align*}
    &\int_{-1}^1G^*_{12}(t+\alpha u;H)du\\
    =&\sum_{h=-\infty}^\infty a(h)\sum_{\begin{subarray}{c}  n_1\leq q_1\sqrt{T}\\ n_1\equiv r_1\!\!\!\!\!\pmod{q_1}\end{subarray}}e\Big(\frac{hq_1t^2}{n_1}-\frac{hr_2}{q_2}\Big)\int_{-1}^1e\Big(\frac{2hq_1t\alpha u+hq_1\alpha^2u^2}{n_1}\Big)du\\
    \ll&|a(0)|\sqrt{T}+\sum_{h=1}^\infty |a(h)|\sum_{\begin{subarray}{c}  n_1\leq q_1\sqrt{T}\\ n_1\equiv r_1\!\!\!\!\!\pmod{q_1}\end{subarray}}\frac{n_1}{hq_1t\alpha}\\
    \ll&H^{-1}T^\frac{1}{2}\log H+\sum_{h=1}^HH^{-1}(\log H)T(ht\alpha)^{-1}+\sum_{h=H}^\infty HT(t\alpha)^{-1}h^{-3}\\
    \ll&H^{-1}T^\frac{1}{2}\log^2 H,
\end{align*}
where the first derivative test was used. This estimate remain valid with $G^*_{12}$ replaced by $G^*_{21}$, which yields
\begin{equation}\label{G}
    \int_{-1}^1G^*(t+\alpha u)du\ll H^{-1}T^\frac{1}{2}\log^2 H.
\end{equation}

Now we  estimate the integral $\int_{-1}^1R^*(t+\alpha u)K_\zeta(u)du$. Let $\theta_0$ be some constant. By the elementary formula
\begin{align*}
    &\cos\big(4\pi (t+\alpha u)\sqrt{n}+\theta_0\big)\\
    =&\cos\big(4\pi t\sqrt{n}+\theta_0\big)\cos(4\pi \alpha u\sqrt{n})-\sin\big(4\pi t\sqrt{n}+\theta_0\big)\sin(4\pi \alpha u\sqrt{n}),
\end{align*}
we get
\begin{align*}
    \int_{-1}^1\cos\big(4\pi (t+\alpha u)\sqrt{n}+\theta_0\big)(1-|u|)\big(1+\zeta\sin(4\pi\alpha u)\big)du
    =I_1-I_2,
\end{align*}
with
\begin{align*}
    I_1=&\cos\big(4\pi t\sqrt{n}+\theta_0\big)\int_{-1}^1\cos(4\pi \alpha u\sqrt{n})(1-|u|)\big(1+\zeta\sin(4\pi\alpha u)\big)du\\
      =&\cos\big(4\pi t\sqrt{n}+\theta_0\big)\int_{-1}^1\cos(4\pi \alpha u\sqrt{n})(1-|u|)du,
 \end{align*}
  \begin{align*}
    I_2=&\sin\big(4\pi t\sqrt{n}+\theta_0\big)\int_{-1}^1\sin(4\pi \alpha u\sqrt{n})(1-|u|)\big(1+\zeta\sin(4\pi\alpha u)\big)du\\
    =&\zeta\sin\big(4\pi t\sqrt{n}+\theta_0\big)\int_{-1}^1\sin(4\pi \alpha u\sqrt{n})(1-|u|)\sin(4\pi\alpha u)du\\
   =&\frac{\zeta}{2}\sin\big(4\pi t\sqrt{n}+\theta_0\big)\int_{-1}^1(1-|u|)\cos\big(4\pi\alpha u(\sqrt{n}-1)\big)du\\
   &-\frac{\zeta}{2}\sin\big(4\pi t\sqrt{n}+\theta_0\big)\int_{-1}^1(1-|u|)\cos\big(4\pi\alpha u(\sqrt{n}+1)\big)du.\\
\end{align*}
By using
\begin{equation*}
     \int_0^1(1-u)\cos(A u)du\ll |A|^{-2}\quad A\neq0,
\end{equation*}
we have
\begin{align*}
    I_1\ll&\alpha^{-2}n^{-1},\\
    I_2=&\left\{\begin{array}{ll}\frac{\zeta}{2}\sin\big(4\pi t\!+\theta_0\big)+O(\alpha^{-2}),&n=1, \\O(\alpha^{-2}(\sqrt{n}-1)^{-2}),&n\neq 1,\end{array}\right.
\end{align*}
which suggests
\begin{align*}
    \int_{-1}^1\cos\big(4\pi (t+\alpha u)\sqrt{n}\!+\!\theta_0\big)K_\zeta(u)du=&\left\{\begin{array}{ll}-\frac{\zeta}{2}\sin\big(4\pi t\!+\theta_0\big)+O(\alpha^{-2}),&n=1, \\O(\alpha^{-2}(\sqrt{n}-1)^{-2}),&n\neq 1.\end{array}\right.
\end{align*}
Take $H= T$, $y= T^\frac{1}{2}$. Then clearly $y>1$. Thus we get
\begin{align}\label{R}
   &\int_{-1}^1R^*(t+\alpha u)K_\zeta(u)du\\
\nonumber=&\!-\frac{\zeta}{2}\sin\bigg(4\pi t -2\pi\Big(\frac{r_2}{q_2}+\frac{r_1}{q_1}+\frac{1}{8}\Big)\bigg)+O\Big(\sum_{n>1}\frac{d(n)}{\alpha^2n^\frac{3}{4}(\sqrt{n}-1 )^{2}}\Big)\\
\nonumber=&\!-\frac{\zeta}{2}\sin\bigg(4\pi t -2\pi\Big(\frac{r_2}{q_2}+\frac{r_1}{q_1}+\frac{1}{8}\Big)\bigg) +O(\alpha^{-2}),
\end{align}
by using $\sum_{n>1}\frac{d(n)}{n^\frac{3}{4}(\sqrt{n}-1)^{2}}\ll1$.
Noting that $H=T$, $t\asymp T^\frac{1}{2}$, by \eqref{s}-\eqref{R}, we see
\begin{align*}
   &\int_{-1}^1\Delta^{**}(t+\alpha u)K_\zeta(u)du\\
=&-\frac{\zeta}{2}\sin\bigg(4\pi t -2\pi\Big(\frac{r_2}{q_2}+\frac{r_1}{q_1}+\frac{1}{8}\Big)\bigg) +O(\alpha^{-2})\\
&+O\big(t^{-\frac{1}{2}}\sup_{|u|\leq1}f((t\!+\!\alpha u)^2)\big)\!+\!O\big(t^{-\frac{1}{2}}H^{-1}T^\frac{1}{2}\mathcal{L}^2\big)
  \!+\!O\big(t^{-\frac{1}{2}}\mathcal{L}^3\big).
\end{align*}
Thus we complete the proof of Lemma \ref{lem:change}
\end{proof}

\section{ The mean value of $\Delta(q_1q_2x; r_1, q_1, r_2, q_2)$ in short intervals}

In this section, we need the following Lemma.

\begin{lemma}\label{lem:Hil}$($ Hilbert's inequality $)$$($ See e.g.\cite{shan} $)${\, \, \rm} Let $x_1<x_2<\cdots<x_n$ be a sequence of real numbers.  If there exists$ ~\delta>0$, such that $\min\limits_{s\neq r}|x_r-x_s|\geq \delta_r\geq\delta>0 (1\leq r\leq n)$, then there exists an absolute constant $C$, such that
\begin{equation*}
  \bigg|\sum_{s\neq r}u_r\bar{{u_s}}(x_r-x_s)^{-1}\bigg|\leq C\sum_r {\delta_r}^{-1}\left|u_r\right|^2,
\end{equation*}
for arbitrary complex numbers $u_1,u_2,\cdots,u_n$.
\end{lemma}

Suppose $T\gg (q_1q_2)^{\varepsilon}$ is a large parameter, $1\leq h_0\leq \frac{1}{2}\sqrt{T}$. Denote $\Delta^{*}(q_1q_2x)=\Delta(q_1q_2x; r_1, q_1, r_2, q_2)$.  In this section we shall
estimate the integral
\begin{equation*}
    I(T,h_0)=\int_1^T\big(\Delta^{*}(q_1q_2(x+h_0)-\Delta^{*}(q_1q_2x)\big)^2dx,
\end{equation*}
which would play an important role in the proof of Theorem \ref{th:maintain}. This type of integral was studied for the error term in the mean square of $\zeta(\frac{1}{2}+ it)$ by Good \cite{good1977einomega},
for the error term in the Dirichlet divisor problem by Jutila \cite{jutila1984divisor} and for the error term in Weyl's law for Heisenberg manifold by Tsang and Zhai \cite{tsang2012sign}. Here we follows the approach of Tsang and Zhai \cite{tsang2012sign} and prove the following
\begin{lemma}\label{lem:meanvalue}  The estimate
\begin{equation*}
    I(T,h_0)\ll Th_0\log^3\frac{\sqrt{T}}{h_0}+T\mathcal{L}^6
\end{equation*}
holds  uniformly for $1\leq h_0\leq \frac{1}{2}\sqrt{T}$.
\end{lemma}

\begin{proof}
Write
\begin{equation}\label{I}
    I(T,h_0)=\int_1+\int_2,
\end{equation}
where
\begin{align*}
    \int_1=&\int_1^{100\max(h_0^2,T^\frac{2}{3})}(\Delta^{*}(q_1q_2(x+h_0)-\Delta^{*}(q_1q_2x)\big)^2dx,\\
    \int_2=&\int_{100\max(h_0^2,T^\frac{2}{3})}^T(\Delta^{*}(q_1q_2(x+h_0)-\Delta^{*}(q_1q_2x)\big)^2dx.
\end{align*}
From Corollary \ref{cor:1}, we see that
\begin{equation}\label{int1}
    \int_1\ll (h_0^3+T)\ll Th_0.
\end{equation}

For $\int_2$, first we estimate the integral
\begin{equation}\label{J}
    J(U,h_0)\!=\!\int_U^{2U}\!\!\!(\Delta^{*}(q_1q_2(x+h_0)-\Delta^{*}(q_1q_2x)\big)^2dx, \!\quad100\max(h_0^2,T^\frac{2}{3})\leq U\leq T.
\end{equation}
Let $T=2U$ in \eqref{vor0}. Then
\begin{align*}
   \Delta^{*}(q_1q_2x)
=&R_0(x; y)\!+\!R_{12}(x; y, H)\!+\!R_{21}(x; y, H)\!\\
\nonumber&+\!G_{12}(x; H)\!+\!G_{21}(x; H)
  \!+\!O\big(\log^3U\big).
\end{align*}
Take $H=U$, $y=\min\big(\frac{1}{2}Uh_0^{-1},U\log^{-6} U\big)$. From \cite[Lemma 4.1 and eq.(4.11)]{LiuKui}, we see
\begin{gather*}
\int_U^{2U}|G_{12}(x; H)\!+\!G_{21}(x; H)|^2dx\ll U\log U,\label{G2}\\
\int_U^{2U}|R_{12}(x; y, H)\!+\!R_{21}(x; y, H)|^2dx\ll U^\frac{3}{2}y^{-\frac{1}{2}}\log^3 U.\label{R2}
\end{gather*}
Therefor
\begin{align}\label{s/r}
    \int_U^{2U}\big(\Delta^{*}(q_1q_2x)-\!R_{0}(x; y)\big)^2dx\ll&  U^\frac{3}{2}y^{-\frac{1}{2}}\log^3 U+ U\log^6 U\\
   \nonumber \ll&  Uh_0^\frac{1}{2}\log^3 U+ U\log^6 U.
\end{align}

We now estimate $\int_U^{2U}\big(R_{0}(x+h_0; y)-R_{0}(x; y)\big)^2dx$. Set $\theta(h,l)=2\pi(\frac{hr_2}{q_2}+\frac{lr_1}{q_1})$. From \eqref{R_0}, we have
\begin{equation}\label{dR0}
    R_{0}(x+h_0; y)-R_{0}(x; y)=F_1(x)+F_2(x),
\end{equation}
where
\begin{align*}
    F_1(x)\!=&\frac{1}{\sqrt{2}\pi}\big((x\!+\!h_0)^\frac{1}{4}\!-\!x^\frac{1}{4}\big)\sum_{n\leq y}\!\frac{1}{n^{\frac{3}{4}}}\!\sum_{n=hl}\!\cos\!\big(\!4\pi \sqrt{n(x\!+\!h_0)}\!-\theta(h,l)\!-\!\frac{\pi}{4}\big)\!\Big),\\
    F_2(x)\!=&\frac{x^\frac{1}{4}}{\sqrt{2}\pi}\!\sum_{n\leq y}\!\frac{1}{n^{\frac{3}{4}}}\!\sum_{n=hl}
\!\!\Big(\!\cos\!\big(\!4\pi\! \sqrt{n(x\!+\!h_0)}\!-\!\theta(h,l)\!-\!\frac{\pi}{4}\!\big)\!-\!\cos\big (4\pi\! \sqrt{nx}\!-\!\theta(h,l)\!-\!\frac{\pi}{4}\!\big)\!\Big).
\end{align*}

From \cite[Proof of Lemma 4.2]{LiuKui}, we get
\begin{align}\label{F1}
   \int_U^{2U}F_1^2(x)dx\ll h_0^2U^{-2}\int_U^{2U}R_0^2(x+h_0)dx\ll h_0^2U^{-\frac{1}{2}}.
\end{align}

For  the mean square of $F_2(x)$, we see
\begin{equation}\label{F20}
    F_2^2=F_{21}+F_{22},
\end{equation}
where
\begin{align*}
    F_{21}(x)\!
    =&\frac{x^\frac{1}{2}}{2\pi^2}\!\sum_{n\leq y}\!\frac{1}{n^{\frac{3}{2}}}\\
&\times\Big(\sum_{n=hl}\!\cos\!\big(\!4\pi \sqrt{n(x\!+\!h_0)}\!-\!\theta(h,l)\!-\!\frac{\pi}{4}\big)\!-\!\cos\big (4\pi \sqrt{nx}\!-\!\theta(h,l)\!-\!\frac{\pi}{4}\!\big)\!\Big)^2,\\
    F_{22}(x)\!
    =&\frac{x^\frac{1}{2}}{2\pi^2}\!\sum_{\begin{subarray}{c}n_1,n_2\leq y\\n_1\neq n_2\end{subarray}}\!\frac{1}{(n_1n_2)^{\frac{3}{4}}}\sum_{n_1=h_1l_1}\sum_{n_2=h_2l_2}\\
&\times\Big(\!\cos\!\big(\!4\pi \sqrt{n_1(x\!+\!h_0)}\!-\!\theta(h_1,l_1)\!-\!\frac{\pi}{4}\!\big)\!-\!\cos\big (4\pi \sqrt{n_1x}\!-\!\theta(h_1,l_1)\!-\!\frac{\pi}{4}\!\big)\!\Big)\\
&\times\Big(\!\cos\!\big(\!4\pi \sqrt{n_2(x\!+\!h_0)}\!-\!\theta(h_2,l_2)\!-\!\frac{\pi}{4}\!\big)\!-\!\cos\big (4\pi \sqrt{n_2x}\!-\!\theta(h_2,l_2)\!-\!\frac{\pi}{4}\!\big)\!\Big)\\
=&\frac{x^\frac{1}{2}}{2\pi^2}\!\sum_{\begin{subarray}{c}n_1,n_2\leq y\\n_1\neq n_2\end{subarray}}\!\frac{1}{(n_1n_2)^{\frac{3}{4}}}\sum_{n_1=h_1l_1}\sum_{n_2=h_2l_2}\sum_{j_1=0}^1 \sum_{j_2=0}^1(-1)^{j_1+j_2}\\
& \!\times\!\cos\!\big(\!4\pi \sqrt{n_1(x\!+\!j_1h_0)}\!-\!\theta(h_1,l_1)\!-\!\frac{\pi}{4}\!\big)\!\cos\!\big(\!4\pi \sqrt{n_2(x\!+\!j_2h_0)}\!-\!\theta(h_2,l_2)\!-\!\frac{\pi}{4}\!\big).
\end{align*}

Write
 \begin{align}\label{F220}
    F_{22}(x)=:F_{221}(x)+F_{222}(x),
 \end{align}
with
\begin{align*}
    F_{221}(x)=&\frac{x^\frac{1}{2}}{4\pi^2}\sum_{j_1=0}^1 \sum_{j_2=0}^1(-1)^{j_1+j_2}\sum_{\begin{subarray}{c}n_1,n_2\leq y\\n_1\neq n_2\end{subarray}}\frac{1}{(n_1n_2)^{\frac{3}{4}}}\sum_{n_1=h_1l_1}\sum_{n_2=h_2l_2}\\
& \times\cos\big(4\pi \sqrt{n_1(x+j_1h_0)}-4\pi \sqrt{n_2(x+j_2h_0)}-\theta(h_1-h_2,l_1-l_2)\big),
 \end{align*}
\begin{align*}
    F_{222}(x)=&\frac{x^\frac{1}{2}}{4\pi^2}\sum_{j_1=0}^1 \sum_{j_2=0}^1(-1)^{j_1+j_2}\sum_{\begin{subarray}{c}n_1,n_2\leq y\\n_1\neq n_2\end{subarray}}\frac{1}{(n_1n_2)^{\frac{3}{4}}}\sum_{n_1=h_1l_1}\sum_{n_2=h_2l_2}\\
& \times\sin\big(4\pi \sqrt{n_1(x+j_1h_0)}+4\pi \sqrt{n_2(x+j_2h_0)}-\theta(h_1+h_2,l_1+l_2)\big).
\end{align*}
Let
\begin{equation*}
    g_\pm(x)=4\pi\sqrt{n_1(x\!+\!j_1h_0)}\pm\!4\pi\sqrt{n_2(x\!+\!j_2h_0)}-\theta(h_1\pm h_2,l_1\pm l_2).
\end{equation*}
Using
\begin{equation*}
    (1+t)^\frac{1}{2}=1+\sum_{v=1}^\infty d_vt^v\quad\big(|t|\leq\frac{1}{2}\big),
\end{equation*}
with $|d_v|<1$, we see
\begin{equation*}
    g_\pm(x)=4\pi\sqrt{x}(\sqrt{n_1}\pm\sqrt{n_2})+4\pi\sum_{v=1}^\infty \frac{d_vh_0^v}{x^{v-\frac{1}{2}}}(\sqrt{n_1}j_1^v\pm\sqrt{n_2}j_2^v)-\theta(h_1\pm h_2,l_1\pm l_2).
\end{equation*}
Noting that $n_1,n_2\leq y\leq \frac{1}{2}Uh_0^{-1}$, we have
\begin{equation*}
    |g'_\pm(x)|\gg \frac{1}{\sqrt{x}}|\sqrt{n_1}\pm\sqrt{n_2}|\quad(n_1\neq n_2).
\end{equation*}
Then by the  the first derivative test we get
\begin{align*}
    \int_U^{2U}F_{221}(x)dx\ll& U\sum_{\begin{subarray}{c}n_1,n_2\leq y\\n_1\neq n_2\end{subarray}}\frac{1}{(n_1n_2)^{\frac{3}{4}}}\sum_{n_1=h_1l_1}\sum_{n_2=h_2l_2}\frac{1}{|\sqrt{n_1}-\sqrt{n_2}|}\\
    =& U\sum_{\begin{subarray}{c}n_1,n_2\leq y\\n_1\neq n_2\end{subarray}}\frac{1}{(n_1n_2)^{\frac{3}{4}}}\frac{d(n_1)d(n_2)}{|\sqrt{n_1}-\sqrt{n_2}|},\\
    \int_U^{2U}F_{222}(x)dx\ll&  U\sum_{\begin{subarray}{c}n_1,n_2\leq y\\n_1\neq n_2\end{subarray}}\frac{1}{(n_1n_2)^{\frac{3}{4}}}\frac{d(n_1)d(n_2)}{|\sqrt{n_1}+\sqrt{n_2}|}.
\end{align*}
Noting $\sum_{n\leq N}d^2(n)\ll N\log^3 N$, by using Lamma \ref{lem:Hil} and \eqref{F220}, we obtain
\begin{align}\label{F22}
    \int_U^{2U}F_{22}(x)dx\ll&U\sum_{\begin{subarray}{c}n_1,n_2\leq y\\n_1\neq n_2\end{subarray}}\frac{1}{(n_1n_2)^{\frac{3}{4}}}\frac{d(n_1)d(n_2)}{|\sqrt{n_1}-\sqrt{n_2}|} \ll U\log^4y.
\end{align}

By the  elementary formulas
\begin{align*}
    \cos u\!-\!\cos v\!=\!-2\sin\big(\frac{u+v}{2}\big)\sin\big(\frac{u\!-\!v}{2}\big),~\text{and}~  \sin(u\!-\!v)=\sin u\cos v\!-\!\cos u\sin v,
\end{align*}
we have
\begin{align}\label{F210}
   F_{21}(x)\!
    =&\frac{2x^\frac{1}{2}}{\pi^2}\!\sum_{n\leq y}\!\frac{1}{n^{\frac{3}{2}}}\sin^2\!\big(\!2\pi \sqrt{n(x\!+\!h_0)}\!-2\pi \sqrt{nx}\big)\\
   \nonumber &\times\Big(\sum_{n=hl}\!\sin\!\big(2\pi \sqrt{n(x\!+\!h_0)}\!+2\pi \sqrt{nx}-\!\theta(h,l)\!-\!\frac{\pi}{4}\big)\Big)^2,\\
   \nonumber =&:F_{211}+F_{212}+F_{213},
\end{align}
where
\begin{align*}
    F_{211}=&\frac{2x^\frac{1}{2}}{\pi^2}\!\sum_{n\leq y}\!\frac{1}{n^{\frac{3}{2}}}\sin^2\!\big(\!2\pi \sqrt{n(x\!+\!h_0)}\!-2\pi \sqrt{nx}\big)\\
    &\times\sin^2\!\big(\!2\pi \sqrt{n(x\!+\!h_0)}\!+2\pi \sqrt{nx}\big)\Big(\sum_{n=hl}\!\cos\!\big(\theta(h,l)\!+\!\frac{\pi}{4}\big)\Big)^2,
\end{align*}
\begin{align*}
    \hspace*{4.1em}F_{212}=&\frac{2x^\frac{1}{2}}{\pi^2}\!\sum_{n\leq y}\!\frac{1}{n^{\frac{3}{2}}}\sin^2\!\big(\!2\pi \sqrt{n(x\!+\!h_0)}\!-2\pi \sqrt{nx}\big)\\
    &\times\cos^2\!\big(\!2\pi \sqrt{n(x\!+\!h_0)}\!+2\pi \sqrt{nx}\big)\Big(\sum_{n=hl}\!\sin\!\big(\theta(h,l)\!+\!\frac{\pi}{4}\big)\Big)^2,\\
    F_{213}=&-\frac{2x^\frac{1}{2}}{\pi^2}\!\sum_{n\leq y}\!\frac{1}{n^{\frac{3}{2}}}\sin^2\!\big(\!2\pi \sqrt{n(x\!+\!h_0)}\!-2\pi \sqrt{nx}\big)\sin\!\big(\!4\pi \sqrt{n(x\!+\!h_0)}\!+4\pi \sqrt{nx}\big)\\
    &\times\sum_{n=hl}\!\sin\!\big(\theta(h,l)\!+\!\frac{\pi}{4}\big) \!\sum_{n=h'l'}\cos\!\big(\theta(h',l')\!+\!\frac{\pi}{4}\big).
\end{align*}
It is easy to see that
\begin{align*}
    0\leq F_{211}+F_{212}
    \leq\frac{2x^\frac{1}{2}}{\pi^2}\!\sum_{n\leq y}\!\frac{1}{n^{\frac{3}{2}}}\sin^2\!\big(\!2\pi \sqrt{n(x\!+\!h_0)}\!-2\pi \sqrt{nx}\big)d^2(n).
\end{align*}
By using Taylor's expansion, we have for $x\geq 100 h_0^2$,
\begin{align*}
    \sin^2\big(2\pi\sqrt{n(x\!+\!h_0)}\!-\!2\pi\sqrt{nx}\big)=&\sin^2\big(\pi h_0n^{\frac{1}{2}}x^{-\frac{1}{2}}+O(h_0^2n^\frac{1}{2}x^{-\frac{3}{2}})\big)\\
    =&\sin^2\big(\pi h_0n^{\frac{1}{2}}x^{-\frac{1}{2}}\big) +O(h_0^2n^\frac{1}{2}x^{-\frac{3}{2}}).
\end{align*}
which suggests
\begin{align*}
    &\int_U^{2U}x^\frac{1}{2}\sin^2\!\big(\!2\pi \sqrt{n(x\!+\!h_0)}\!-2\pi \sqrt{nx}\big)dx\\
    \ll&\int_U^{2U}x^\frac{1}{2}\min\big(1, h_0^2nx^{-1}\big)+O(h_0^2n^\frac{1}{2}x^{-1})dx
    \ll\left\{\begin{array}{ll}U^\frac{1}{2}h_0^2n,&n\leq Uh_0^{-2},\\U^{\frac{3}{2}},&n> Uh_0^{-2},\end{array}\right.
\end{align*}
in view of the fact $h_0^2<U$ and $n\leq y<U$. Hence,
\begin{align}\label{F2112}
 \int_U^{2U}\!\!F_{211}\!+\!F_{212}dx\!\ll h_0^2 U^\frac{1}{2}\!\!\sum_{n\leq Uh_0^{-2}}\!\frac{d^2(n)}{n^{\frac{1}{2}}}\!+U^{\frac{3}{2}}\!\!\sum_{n> Uh_0^{-2}}\!\frac{d^2(n)}{n^{\frac{3}{2}}}\!\ll U h_0\log^3\!\frac{\sqrt{U}}{h_0},
\end{align}
where we used the well-known estimate $\sum_{n\leq N}d^2(n)\ll N\log^3N$.

 By the first derivative test, we have
 \begin{equation*}
    L_n(t):=\int_U^tx^\frac{1}{2} \sin\!\big(\!4\pi \sqrt{n(x\!+\!h_0)}\!+4\pi \sqrt{nx}\big)dx \ll Un^{-\frac{1}{2}},\quad U\leq t\leq2U.
 \end{equation*}
Using the integration by parts, we obtain
\begin{align*}
    &\int_U^{2U}\!\!\!x^\frac{1}{2} \sin^2\big(2\pi\sqrt{n(x\!+\!h_0)}\!-\!2\pi\sqrt{nx}\big)\sin\!\big(\!4\pi \sqrt{n(x\!+\!h_0)}\!+4\pi \sqrt{nx}\big)dx\\
    =&\int_U^{2U}\!\!\!\sin^2\big(2\pi\sqrt{n(x\!+\!h_0)}\!-\!2\pi\sqrt{nx}\big)dL_n(x)\\
    =&L_n(2U)\sin^2\big(2\pi\sqrt{n(2U\!+\!h_0)}\!-\!2\pi\sqrt{2nU}\big)- 2\int_U^{2U}\!L_n(x) \\
    &\times\sin\big(2\pi\sqrt{n(x\!+\!h_0)}\!-\!2\pi\sqrt{nx}\big)\cos\big(2\pi\sqrt{n(x\!+\!h_0)}\!-\!2\pi\sqrt{nx}\big) \big(\frac{\pi\sqrt{n}}{\sqrt{x\!+\!h_0}} \!-\!\frac{\pi\sqrt{n}}{\sqrt{x}}\big)dx\\
    \ll&Un^{-\frac{1}{2}}+U^{\frac{1}{2}}h_0,
\end{align*}
which yields
\begin{align}\label{F213}
    \int_U^{2U}\!\!F_{213}dx\!=\!&-\!\frac{2}{\pi^2}\!\sum_{n\leq y}\!\frac{1}{n^{\frac{3}{2}}}\sum_{n=hl}\!\sin\!\big(\theta(h,l)\!+\!\frac{\pi}{4}\big) \!\sum_{n=h'l'}\cos\!\big(\theta(h',l')\!+\!\frac{\pi}{4}\big)\int_U^{2U}\!\!\!x^\frac{1}{2}\\
    \nonumber &\times\sin^2\!\big(\!2\pi \sqrt{n(x\!+\!h_0)}\!-\!2\pi \sqrt{nx}\big)\sin\!\big(\!4\pi \sqrt{n(x\!+\!h_0)}\!+\!4\pi \sqrt{nx}\big)dx\\  \nonumber\ll&\sum_{n\leq y}\!\frac{d^2(n)}{n^{\frac{3}{2}}}(Un^{-\frac{1}{2}}+U^{\frac{1}{2}}h_0)\ll U.
\end{align}

From \eqref{F210}-\eqref{F213}, we get
\begin{equation}\label{F21}
    \int_U^{2U}F_{21}(x)dx\ll U h_0\log^3\!\frac{\sqrt{U}}{h_0}.
\end{equation}

Combining \eqref{F20}, \eqref{F22} and \eqref{F21}, we obtain
\begin{equation*}
    \int_U^{2U}F_{2}^2(x)dx\ll U h_0\log^3\!\frac{\sqrt{U}}{h_0}+U\log^4y,
\end{equation*}
which together with \eqref{dR0}, \eqref{F1} yields
\begin{equation}\label{dR}
    \int_U^{2U}\big(R_{0}(x+h; y)-R_{0}(x; y)\big)^2(x)dx\ll U h_0\log^3\!\frac{\sqrt{U}}{h_0}+U\log^4y.
\end{equation}

From \eqref{J}, \eqref{s/r}, and \eqref{dR}, it follows that
\begin{equation*}
    J(U,h_0)\ll U h_0\log^3\!\frac{\sqrt{U}}{h_0}+U\log^6y,
\end{equation*}
which implies
\begin{equation}\label{int2}
    \int_2\ll T h_0\log^3\!\frac{\sqrt{T}}{h_0}+T\mathcal{L}^6,
\end{equation}
via a splitting argument. Then Lemma \ref{lem:meanvalue} follows from \eqref{I}, \eqref{int1}, and \eqref{int2}.
\end{proof}
\vspace{2ex}

\section{Proof of Theorem \ref{th:maintain}}

In this section, we will give a proof of Theorem \ref{th:maintain} by following the approach of \cite{tsang2012sign}. We still write $\Delta^{*}(q_1q_2x)=\Delta(q_1q_2x; r_1, q_1, r_2, q_2)$. Define
\begin{align*}
    \Delta^{*}_+(t)=&\frac{1}{2}\big(|\Delta^{*}(t)|+\Delta^{*}(t)\big),&
    \Delta^{*}_-(t)=&\frac{1}{2}\big(|\Delta^{*}(t)|-\Delta^{*}(t)\big).
\end{align*}
We need the following two lemmas.
\begin{lemma}\label{lem:3}
 \begin{equation*}
    \int_T^{2T}{\Delta^{*}}^2_\pm(q_1q_2t)dt\gg T^\frac{3}{2}.
 \end{equation*}
\end{lemma}
\begin{proof}
From Corollary \ref{cor:1} with $k=2,4$, by H\"{o}lder's inequality, we get
\begin{align*}
    T^\frac{3}{2}\ll\int_T^{2T}{\Delta^{*}}^2(q_1q_2t)dt \ll&\Big(\int_T^{2T}|\Delta^{*}(q_1q_2t)|dt\Big)^\frac{2}{3}\Big(\int_T^{2T}{\Delta^{*}}^4(q_1q_2t)dt\Big)^\frac{1}{3}\\
    \ll&\Big(\int_T^{2T}|\Delta^{*}(q_1q_2t)|dt\Big)^\frac{2}{3}T^\frac{2}{3},
\end{align*}
which yields
\begin{equation}\label{31}
    \int_T^{2T}|\Delta^{*}(q_1q_2t)|dt\gg T^\frac{5}{4}.
\end{equation}
From \eqref{muller}, we see
\begin{equation*}
        \int_T^{2T}\Delta^{*}(q_1q_2t)dt\ll T^\frac{3}{4}.
\end{equation*}
Thus, from the definition of $\Delta^{*}_\pm(q_1q_2t)$, we have
\begin{equation*}
    \int_T^{2T}\Delta^{*}_\pm(q_1q_2t)dt\gg T^\frac{5}{4}.
\end{equation*}
Then by  Cauchy-Schwarz's inequality, we get
\begin{align*}
    T^\frac{5}{4}\ll\Big(\int_T^{2T}dt\Big)^\frac{1}{2} \Big(\int_T^{2T}{\Delta^{*}}^2_\pm(q_1q_2t)dt\Big)^\frac{1}{2}\ll T^\frac{1}{2}\Big(\int_T^{2T}{\Delta^{*}}^2_\pm(q_1q_2t)dt\Big)^\frac{1}{2},
\end{align*}
which immediately implies Lemma \ref{lem:3}.
\end{proof}

\begin{lemma}\label{lem:4}
Suppose $2\leq H_0\leq\sqrt{T}$. Then
\begin{equation*}
    \int_T^{2T}\max_{h\leq H_0}\big(\Delta^{*}_\pm(q_1q_2(t+h))-\Delta^{*}_\pm(q_1q_2t)\big)^2 dt\ll H_0T\mathcal{L}^7.
\end{equation*}
\end{lemma}

\begin{proof}
Since
\begin{equation*}
    |\Delta^{*}_\pm(q_1q_2(t+h))-\Delta^{*}_\pm(q_1q_2t)|\leq |\Delta^{*}(q_1q_2(t+h))-\Delta^{*}(q_1q_2t)|,
\end{equation*}
it is sufficient to prove that
\begin{equation*}
   I= \int_T^{2T}\max_{h\leq H_0}\big(\Delta^{*}(q_1q_2(t+h))-\Delta^{*}(q_1q_2t)\big)^2 dt\ll H_0T\mathcal{L}^7.
\end{equation*}

For $0<u_1<u_2\ll T$, it easy to see that
\begin{equation*}
    \Delta^{*}(q_1q_2u_2)-\Delta^{*}(q_1q_2u_1)\geq - O\big((u_2-u_1)\log T\big).
\end{equation*}
Write $H_0=2^\lambda b$, such that $\lambda\in \mathbb{N}$ and $1\leq b <2$. Then for each $t\in[T, 2T]$, we have
 \begin{equation*}
    \max_{h\leq H_0}\big|\Delta^{*}(q_1q_2(t+h))-\Delta^{*}(q_1q_2t)\big|\ll \max_{1\leq j\leq 2^\lambda}\big|\Delta^{*}(q_1q_2(t+jb))-\Delta^{*}(q_1q_2t)\big|+\mathcal{L}.
 \end{equation*}
Similar to the argument of the proof of Lemma 2
of \cite{heathbrown1994sign}, by using Lemma \ref{lem:meanvalue}, we we can deduce that
\begin{align*}
    I\ll& \lambda\sum_{\mu\leq\lambda}\sum_{0\leq\nu\leq2^\mu} \int_{T+\nu2^{\lambda-\mu}b}^{2T+\nu2^{\lambda-\mu}b} \big(\Delta^{*}(q_1q_2(t+2^{\lambda-\mu}b))-\Delta^{*}(q_1q_2t)\big)^2dt\!+\! T\mathcal{L}^2\\
    \ll& \lambda\sum_{\mu\leq\lambda}\sum_{0\leq\nu\leq2^\mu} \big( 2^{\lambda-\mu}bT\mathcal{L}^3+ T\mathcal{L}^6\big)\\
    \ll& \lambda\sum_{\mu\leq\lambda}\big( 2^{\lambda}bT\mathcal{L}^3+ 2^\mu T\mathcal{L}^6\big)\\
    \ll& \lambda^2 H_0T\mathcal{L}^3+\lambda H_0T\mathcal{L}^6\\
    \ll&  H_0T\mathcal{L}^7.
\end{align*}
Thus we get Lemmma \ref{lem:4}.\end{proof}

Now we finish the proof of Theorem \ref{th:maintain}.
Let $P(t)=\Delta^{*}_\pm(q_1q_2t)$ and $Q(t)=\delta t^\frac{1}{4}$ for a  sufficiently small $\delta>0$, and
\begin{align*}
    \omega(t)=P^2(t)-4\max_{h\leq H_0}\big(P(t+h)-P(t)\big)^2-Q^2(t).
\end{align*}
Then
\begin{align}\label{35}
    \int_T^{2T}\!\!\!\!\omega(t)dt\gg T^\frac{3}{2}\!-\!O\big( H_0T\mathcal{L}^7\big)\! -\!O\big(\delta^2 T^\frac{3}{2}\big)\!\gg\!  T^\frac{3}{2},
\end{align}
from Lemma \ref{lem:3} and Lemma \ref{lem:4}, by taking $H_0=\delta T^\frac{1}{2}\mathcal{L}^{-7}$.
For any point $t_0$, where $\omega(t_0)>0$ and any $h\in [0, H_0]$, we see that $P(t_0+h)$  has the same sign as $P(t_0)$, and $|P(t_0+h)|>\frac{1}{2}|Q(t_0)|$.

 Let $$\mathscr{S}=\{t\in[T,2T]:\omega(t)>0\}.$$ From Corollary \ref{cor:1} and \eqref{35}, using Cauchy-Schwarz's inequality, we have
\begin{align*}
     T^\frac{3}{2}\ll & \int_T^{2T}\omega(t)dt\leq \int_\mathscr{S}\omega(t)dt \leq \int_\mathscr{S}{\Delta^{*}}^2_\pm(q_1q_2t)dt \\ \leq&|\mathscr{S}|^\frac{1}{2}\Big(\int_T^{2T}{\Delta^{*}}^4(q_1q_2t)dt\Big)^\frac{1}{2} \ll |\mathscr{S}|^\frac{1}{2} T,
\end{align*}
which implies
\begin{equation*}
    |\mathscr{S}|\gg T.
\end{equation*}
Thus the proof of Theorem \ref{th:maintain} is completed.
\qed

\section{Proof of Theorem \ref{th:omega}}

Suppose $k\geq3$ is a fixed odd integer and $T\gg (q_1q_2)^{\varepsilon}$ is a large parameter.
Set
\begin{equation*}
    \delta=\left\{\begin{array}{ll}-1,&\quad\text{if } C_k\geq0,\\1,&\quad\text{if } C_k<0,\end{array}\right.
\end{equation*}
where $C_k$ is defined in \eqref{sk}.

By Theorem \ref{th:maintain}, there exists $t\in[T,2T]$ such that $\delta \Delta(q_1q_2u; r_1, q_1, r_2, q_2)> c_5 t^\frac{1}{4}$ for any $u\in[t,t+H_0]$, with $H_0=c_4\sqrt{T}\mathcal{L}^{-7}$. Thus
\begin{align*}
    &c^k_5 H_0 t^\frac{k}{4}<\int_t^{t+H_0}\delta^k\Delta^k(q_1q_2u; r_1, q_1, r_2, q_2)du\\
    =&\delta^kC_k\big((t\!+\!\!H_0)^{1\!+\!\frac{k}{4}}\!\!-\!t^{1\!+\!\frac{k}{4}}\big)\!\!+\!\delta^k\!\big( \mathcal{F}_k\big(q_1q_2(t\!+\!\!H_0); r_1, q_1, r_2, q_2\big)\!\!- \!\!\mathcal{F}_k(q_1q_2t; r_1, q_1, r_2, q_2)\big),
\end{align*}
which yields
\begin{align*}
    &\delta^k\Big( \!\mathcal{F}_k\big(q_1q_2(t\!+\!H_0); r_1, q_1, r_2, q_2\big)\!\!- \! \mathcal{F}_k\big(q_1q_2t; r_1, q_1, r_2, q_2\big)\Big)\\
    >&c^k_5 H_0 t^\frac{k}{4}-\delta^kC_k\big(1\!+\!\frac{k}{4}\big)t^{\frac{k}{4}}H_0\!\!+\!O(H_0^2t^{\frac{k}{4}\!-\!1})
    =C_k^*H_0t^{\frac{k}{4}}\big(1+O(H_0T^{-1})\big),
\end{align*}
with
\begin{equation*}
    C_k^*=c^k_5-\delta^kC_k\big(1\!+\!\frac{k}{4}\big)>0.
\end{equation*}
Thus we get
\begin{equation*}
    \big|\mathcal{F}_k\big(q_1q_2(t\!+\!H_0);r_1, q_1, r_2, q_2\big)\!\!- \! \mathcal{F}_k\big(q_1q_2t; r_1, q_1, r_2, q_2\big)\big|\gg H_0T^{\frac{k}{4}},
\end{equation*}
which immediatly implies Theorem \ref{th:omega}.
\qed

\vspace{5ex}
%²Î¿¼ÎÄÏ×

\bibliography{jia2}
\bibliographystyle{plain}

\end{document}